\documentclass[a4paper,11pt]{amsart}

\usepackage{amsmath}
\usepackage{amsthm}
\usepackage{mathtools}
\usepackage{amssymb}
\usepackage{amsfonts}
\usepackage{xcolor}
\usepackage[bookmarks=true,hyperindex,pdftex,colorlinks,citecolor=blue]{hyperref}
\usepackage{enumitem}
\usepackage{marginnote}
\usepackage{soul}
\usepackage{mathrsfs}

\usepackage[a4paper,lmargin=3cm,rmargin=3cm,tmargin=4cm,bmargin=4cm,marginparwidth=2.8cm,marginparsep=1mm]{geometry}

\theoremstyle{plain}
\newtheorem{theorem}{Theorem}[section]
\newtheorem*{theorem*}{Theorem}

\newtheorem{lemma}[theorem]{Lemma}
\newtheorem{corollary}[theorem]{Corollary}
\newtheorem{proposition}[theorem]{Proposition}

\theoremstyle{definition}
\newtheorem*{definition*}{Definition}
\newtheorem{definition}[theorem]{Definition}

\theoremstyle{remark}
\newtheorem{remark}[theorem]{Remark}

\numberwithin{equation}{section}

\def \one {\mathsf{1}}

\def \B {\mathcal{B}}

\def \R {\mathbb{R}}

\def \H {\mathcal{H}}
\def \M {\mathcal{M}}
\def \N {\mathbb{N}}

\def \F {\mathcal{F}}

\def \cal {\mathcal}

\def \eps {\varepsilon}

\def \sbs {\subseteq}
\def \sps {\supseteq}

\DeclareMathOperator{\diam}{diam}
\DeclareMathOperator{\Lip}{Lip}

\DeclareMathOperator{\supp}{supp}

\DeclareMathOperator{\spn}{span}

\begin{document}

	\title[Complementation of free spaces in their biduals]{On complementation and $L$-embedding of Lipschitz-free spaces in their biduals}%
	

	\author[Braga]{Bruno M. Braga}
	\address[B. M. Braga]{IMPA, Estrada Dona Castorina 110, 22460-320, Rio de Janeiro, Brazil}
	\email{demendoncabraga@gmail.com}
	\urladdr{\url{https://sites.google.com/site/demendoncabraga}}
	\thanks  {B. M. Braga  was partially supported by FAPERJ, grant E-26/204.317/2025,  by CNPq, grant 303182/2026-1, and by Serrapilheira, grant R-2501-51476.}

	\author[Gartland]{C. Gartland}
	\address[C. Gartland]{Department of Mathematics and Statistics, University of
		North Carolina at Charlotte, Charlotte, North Carolina, USA.}
	\email{cgartla1@charlotte.edu}
	\urladdr{\url{https://chrisgartland.wordpress.com/}}
	\thanks{C. Gartland was partially supported by NSF Award DMS-2555144}
	
	\author[Lancien]{G. Lancien}
	\address[G. Lancien]{Universit\'e Marie et Louis Pasteur, CNRS, Laboratoire de mathématiques de Besançon  (UMR 6623), F-25000 Besan\c con, France.}
	\email{gilles.lancien@umlp.fr}
	\urladdr{\url{https://glancien.pages.math.cnrs.fr/}}
	\thanks{G. Lancien was partially supported by the French ANR project No. ANR-24-CE40-0892-01.}
	
	\author[Motakis]{P. Motakis}
	\address[P. Motakis]{Department of Mathematics and Statistics, York University, 4700 Keele Street, Toronto, Ontario, M3J 1P3, Canada.}
	\email{pmotakis@yorku.ca}
	\urladdr{\url{https://pmotakis.mathstats.yorku.ca/}}
	\thanks{P. Motakis was partially supported by NSERC Grant RGPIN-2021-03639.}
	
	\author[Perneck\'a]{E. Perneck\'a}
	\address[E. Perneck\'a]{Faculty of Information Technology, Czech Technical University in Prague, Th\'akurova 9, 160 00, Prague 6, Czech Republic.}
	\email{eva.pernecka@fit.cvut.cz}
	\urladdr{\url{https://fit.cvut.cz/en/faculty/people/5300-eva-pernecka}}
	\thanks{}
	
	\author[Schlumprecht]{Th. Schlumprecht}
	\address[Th. Schlumprecht]{Texas A\&M University, College Station, TX 77843,
		USA and Faculty of Electrical Engineering, Czech Technical University in Prague, Zikova 4, 166 27, Prague, Czech Republic.}
	\email{t-schlumprecht@tamu.edu}
	\urladdr{\url{https://people.tamu.edu/~t-schlumprecht/}}
	\thanks{Th. Schlumprecht was partially supported by NSF Award DMS-2349322}
	
	\begin{abstract}
		A closed subspace $E$ of a Banach space $F$ is \emph{$L$-embedded in $F$} if there is a projection $P\colon F\to E$ such that $\|x\|=\|Px\|+\|x-Px\|$ for all $x\in F$. We show that if $X$ is a proper, geodesic metric space with locally finite $\H^1$-measure, then its Lipschitz-free space, denoted by $\F(X)$, is $L$-embedded in its bidual. We also extend an argument by C\'uth, Kalenda, and Kaplick\'y, to show that if $E$ is a finite dimensional normed space and $C\ge 1$ is such that  every subspace of $E$ is $C$-complemented in $E$, then $\cal F(E)$ is $C$-complemented in its bidual. In particular, for every two-dimensional normed space $E$, $\cal F(E)$ is $1$-complemented in its bidual.
	\end{abstract}
	
	\maketitle
	
	\section{Introduction}
	For a metric space $M$, we denote by $\cal F(M)$ its Lipschitz-free space (see Section \ref{subsec:LF} for the definition). After the seminal paper by Godefroy and Kalton \cite{GodefroyKalton2003}, the problem of whether a Lipschitz-free space is complemented in its bidual has gained considerable attention. This is motivated by  its potential applications to the Lipschitz classification of Banach spaces: the complementation of $\cal F(\ell_1)$ in its bidual would imply that the Lipschitz geometry of $\ell_1$ characterizes its linear isomorphism type (see \cite[Section 4.7]{BaudierLancienBook}). Until recently, the main result on this subject was obtained by C\'uth, Kalenda, and Kaplick\'y , who  proved that for any $d\in \N$, the space $\cal F(\ell_2^d)$ is $1$-complemented in its bidual (see \cite[Theorem 1.1]{CKK2019}). 
	
	In a very recent preprint \cite{BGLPPS2026}, we have studied a stronger form of complementability called $L$-embeddability (see definition in Section \ref{subsec:L-embeddings} and in the abstract). In particular, we showed that if $X$ is a nonzero Banach space, then its Lipschitz-free space is $L$-embedded in its bidual if and only if $X$ has dimension one (\cite[Theorem A]{BGLPPS2026}). The main result of the present paper shows that, in the general setting of metric spaces, the $L$-embeddability of the Lipschitz-free space in its bidual also seems   to be a one dimensional phenomenon (the relevant definitions for the Hausdorff measure $\H^1$ can be found in \S\ref{subsec:H_1}):
	
	\begin{theorem*}[Theorem~\ref{thm:main}]
		Let $X$ be a proper, geodesic metric space with locally finite $\H^1$-measure. Then $\F(X)$ is $L$-embedded in its bidual, and hence is weakly sequentially complete.
	\end{theorem*}
	
	Since $\cal F(\R)$ is linearly isometric to $L_1(\R)$, Theorem~\ref{thm:main} can be seen as a generalization of the Yosida--Hewitt decomposition, which implies that $L_1(\R)$ is $L$-embedded into its bidual. Our argument is mostly based on a tree parametrization of such metric spaces and fundamental results by Godefroy on subspaces of $L_1$ that are closed both in norm and with respect to the topology of convergence in measure \cite{Godefroy}.
	
	In the last part of the paper, we return to the general complementability problem. We add and adapt a new ingredient to the argument of C\'uth, Kalenda, and Kaplick\'y (\cite[Theorem 1.1]{CKK2019}) to show that if $E$ is a finite dimensional normed space and $C\ge 1$ is such that every subspace of $E$ is $C$-complemented in $E$, then $\cal F(E)$ is $C$-complemented in its bidual (Theorem \ref{t:C-complementation}). As an immediate corollary, we obtain that, for every two-dimensional normed space $E$, $\cal F(E)$ is $1$-complemented in its bidual (see Corollary \ref{CorollaryDim2}). This should be compared to \cite[Theorem A]{BGLPPS2026}, which asserts that $\F(E)$ is never $L$-embedded in its bidual when $\dim E \geq 2$. The new ingredient we use consists of considering the Effros--Borel structure on the set of subspaces of $E$, introduced by Bossard in \cite{Bossard2002}, and a Borel selection of projections of norm at most $C$ from $E$ to a given subspace of $E$ (Proposition \ref{p:Borel0}). \\
	
	\noindent\textbf{AI usage statement.}
	AI tools were neither used to develop the ideas, nor to generate or verify the proofs in this project.
	
	\section{Notation and preliminaries}
	\label{sec:prelims}
	
	\subsection{Banach spaces and $L$-embeddability}\label{subsec:L-embeddings}
	
	In this article, all Banach spaces will be over the field $\R$. 
	Given two Banach spaces $X$ and $Y$, we denote by $B(X,Y)$ the space of all bounded linear  operators from $X$ to $Y$, $B(X):=B(X,X)$, and we write $X^* = B(X, \R)$ the topological dual of $X$. We denote the closed unit ball of $X$ by $B_X$ and its unit sphere by $S_X$. 
	
	A \emph{projection} on a Banach space $X$ is a bounded linear operator  $P\colon X\to X$  with $P^2=P$. If moreover $P$ satisfies 
	\begin{equation}\label{Defi.Lprof.intro}\|x\|=\|Px\|+\|x-Px\|,\ \text{ for all }\ x\in X,
	\end{equation}
	$P$ is called an \emph{$L$-projection} and its image an \emph{$L$-summand}. A closed subspace $E$ of a Banach space $F$ is \emph{$L$-embedded in $F$} if there is an $L$-projection from $F$ onto $E$. 
	
	We will be interested in Banach spaces that are $L$-embedded in their biduals. As mentioned in the introduction, any $L_1$ space is $L$-embedded in its bidual. We also recall an important feature of Banach spaces that are $L$-embedded in their biduals: they are always weakly sequentially complete (\cite[Lemme 4]{Godefroy}).  
	
	\subsection{Metric spaces}
	
	In this very short subsection, we recall a few classical definitions for metric spaces. 
	
	\begin{definition} A metric space $(X,d)$ is called \emph{geodesic} if for every pair of points $x,y \in X$, there exists an isometry   $\gamma: [0,d(x,y)] \to X$, such that $\gamma(0)=x$ and $\gamma(d(x,y))=y$. Such a map  is called a \emph{geodesic path} from $x$ to $y$. 
	\end{definition}
	
	\begin{definition} A metric space $(X,d)$ is called \emph{proper} if every closed and bounded subset of $X$ is compact. As is typical in the literature, we abuse terminology and say that a map between metric spaces is \emph{proper} if the preimage of compact subsets is compact.
	\end{definition}
	
	We now turn to  the notion of unit speed for a Lipschitz curve.
	
	\begin{definition}
		Let us say that a Lipschitz map $\gamma: I \to (X,d)$ from a compact interval $I$ to a metric space $(X,d)$ has {\it unit speed} if for all $a \leq b \in I$, $\ell(\gamma\big|_{[a,b]}) = b-a$, where $\ell(g)$ denotes the length of a Lipschitz curve $g: [c,d] \to X$. 
	\end{definition}
	Note that unit speed Lipschitz maps are necessarily 1-Lipschitz.
	
	\subsection{Lipschitz-free spaces}\label{subsec:LF}
	By a \emph{pointed metric space} we mean a metric space $M$ containing a distinguished point, denoted 0 and called \emph{base point} or \emph{origin}. If $M$ and $N$ are pointed metric spaces then $\Lip_0(M,N)$ stands for the set of all Lipschitz functions $f :M \to N$ such that $f(0) = 0$. If $(M,d)$ is a metric space and $N=X$ is a real Banach space, then $\Lip_0(M,X)$ is a Banach space when equipped with the Lipschitz norm:
	$$ \forall f \in \Lip_0(M,X), \quad \|f\|_L = \sup_{x \neq y \in M} \frac{\|f(x)-f(y)\|_X}{d(x,y)}.$$
	In the case $X = \R$, we denote $\Lip_0(M):=\Lip_0(M,\R)$. Then, for $x\in M$, we let $\delta_x \in \Lip_0(M)^*$ be the evaluation functional defined by $\langle\delta_x , f \rangle  = f(x)$, for all $f \in \Lip_0(M)$. It is readily seen that $\delta_M: x \in M \mapsto \delta_x \in \Lip_0(M)^*$ is an isometry. When no confusion is possible, we will just write $\delta=\delta_M$. The \textit{Lipschitz-free space over $M$} (free space for brevity) is then defined as the closed linear span in $\Lip_0(M)^*$ of all such evaluation functionals:
	$$\F(M) := \overline{ \mbox{span}}^{\| \cdot  \|}\left \{ \delta_x \, : \, x \in M  \right \} \sbs \Lip_0(M)^*.$$
	
	Let us now recall the universal extension property of free spaces. Let  $X$ be a Banach space and $f \in \Lip_{0}(M,X)$. Then there exists a unique $\overline{f} \in B(\F(M),X)$ such that $f = \overline{f} \circ \delta_M$. Moreover $\| \overline{f} \|_{B(\F(M),X)} = \| f \|_L$. Therefore, $\Lip_0(M,X)$ is linearly isometric to $B(\F(M),X)$ and, in particular, $\Lip_0(M)$ is linearly isometric to  $\F(M)^*$. Then the weak$^*$ topology induced by $\F(M)$ on $\Lip_0(M)$ coincides on bounded subsets of $\Lip_0(M)$ with the topology of pointwise convergence. 
	
	Another application of the universal extension property is the following factorization result. Let $M,N$ be pointed metric spaces and let $f:M\to N$ be a Lipschitz map satisfying $f(0)=0$. Then there exists a unique $\widehat f\in B(\F(M),\F(N))$ such that $\widehat f \circ \delta_M=\delta_N \circ f$. Moreover $\| \widehat {f} \|_{B(\F(M),\F(N))} = \Lip(f)$. We shall use this result and the notation $\widehat f$ throughout this note. 
	
	For a classical reference on free spaces, we refer the reader to~\cite{Weaver2}.
	
	\subsection{One dimensional Hausdorff measure} \label{subsec:H_1}
	
	Let $(X,d)$ be a separable metric space. We recall the basic facts that we shall need on $\mathcal H^1$, the $1$-dimensional Hausdorff measure on $X$. We refer to the textbook \cite[Chapter 4]{Mattila_Book} for the proofs (the results in \cite{Mattila_Book} are often stated only for $X = \R^n$, but the proofs of the basic results we require clearly hold in any separable metric space). For a subset $A$ of $X$ and $\delta>0$,
	\[\cal H^1_\delta(A):=\inf\left\{\sum_{i=1}^\infty \diam(C_i): A\sbs \bigcup_{i=1}^\infty C_i,\ \diam(C_i)\le \delta\right\}\ \ \text{and}\ \ \cal H^1(A):=\lim_{\delta \to 0}\cal H^1_\delta(A).\]
	Then $\cal H^1$ is an outer measure on the subsets of $X$. A subset $E$ of $X$ is called \emph{$\cal H^1$-measurable} if for any subset $A$ of $X$, \[\cal H^1(A)=\cal H^1(A \cap E)+\cal H^1(A \setminus E).\] The class $\Sigma_{\cal H^1}(X)$ of $\cal H^1$-measurable subsets of $X$ is a $\sigma$-algebra containing the Borel $\sigma$-algebra $\cal B(X)$. Then $\cal H^1$ is a ($\sigma$-additive) measure on $\Sigma_{\cal H^1}(X)$, which is complete for $\cal H^1$. The outer measure $\cal H^1$ is \emph{Borel regular}, meaning that for every subset $A$ of $X$, there exists $B\in \cal B(X)$ such that $A\sbs B$ and $\cal H^1(A)=\cal H^1(B)$. Assume now that $\cal H^1(X)<\infty$. Then passing to complements and using the definition of $\Sigma_{\cal H^1}(X)$, we get that for every $A \in \Sigma_{\cal H^1}(X)$, there exists $B \in \cal B(X)$ such that $B\sbs A$ and $\cal H^1(B)=\cal H^1(A)$, and again since $\cal H^1(A)<\infty$, this implies that $\cal H^1(A \setminus B)=0$. Importantly, we also recall that the Lebesgue measure on $\R$ coincides with the $1$-dimensional Hausdorff measure on $\R$.
	
	\begin{definition}
		We say that a metric space $X$ has \emph{locally finite $\H^1$-measure} if $\H^1(K) < \infty$ for all compact subsets $K$ of $X$.   
	\end{definition}
	
	\subsection{Metric differentials and Kirchheim's theorems}
	
	We now recall the important notion of metric differentials from the theory of rectifiable subsets of metric spaces, as well as two fundamental theorems of Kirchheim.
	
	\begin{definition}
		Let $A\sbs\R$ and $(X,d)$ be a metric space. The \textit{metric differential} of a map $\gamma: A \to X$ at a limit point $x \in A$ is defined to be the nonnegative number
		$$M\! D(\gamma,x) := \lim_{A\ni y \to x} \dfrac{d(\gamma(x),\gamma(y))}{|x-y|},$$
		whenever this limit exists in $\R$. The set of all limit points $x \in A$ for which $M\! D(\gamma,x)$ exists is denoted $\mathscr{M\! D}(\gamma)$.
	\end{definition}
	
	Given a Borel measure $\mu$ on a metric space $X$ (i.e., a measure whose associated $\sigma$-algebra contains $\cal B(X)$) and a $\mu$-measurable subset $A \sbs X$, we say that a $\mu$-measurable subset $B \sbs X$ has {\it full $\mu$-measure in $A$} if $B \sbs A$ and $\mu(A\setminus B) = 0$. We let $\lambda$ denote the Lebesgue measure on $(\R,\cal B(\R))$ and still let $\H^1$ denote the 1-dimensional Hausdorff measure on a metric space $X$.
	
	\begin{theorem}[Kirchheim Differentiation; {\cite[Theorem 2]{Kirchheim}}]\label{eq:Kirchheimdiff} For any Borel set $A\sbs\R$, any metric space $X$, and any Lipschitz map $\gamma: A \to X$, the set $\mathscr{M\! D}(\gamma)$ has full $\lambda$-measure in $A$.
	\end{theorem}
	
	\begin{theorem}[Area Formula; {\cite[Corollary~8]{Kirchheim}}]\label{thm:areaformula} Let $A\sbs\R$ be a Borel set, $X$ a metric space, and $\gamma\colon A \to X$ a Lipschitz map. For each $y \in X$, let $|\gamma^{-1}(y)| \in \N$ denote the cardinality of the preimage of $\{y\}$ under $\gamma$. Then
		$$\int_A M\! D(\gamma,x)\, d\lambda(x) = \int_{\gamma(A)} |\gamma^{-1}(y)|\, d\H^1(y).$$
	\end{theorem}
	
	Note that in the above statement, the notation $\int_{\gamma(A)} |\gamma^{-1}(y)|\, d\H^1(y)$ makes sense as $\gamma(A)$ is $\H^1$-measurable. Indeed, $A$ is the union of a $\lambda$-null set $N$ and a countable union of compact sets. Since $\gamma$ is continuous, it follows that $\gamma(A)$ is the union of $\gamma(N)$ and a countable union of compact sets. But $\gamma$ is Lipschitz and $\lambda$ is the $\H^1$-measure on $\R$, so it is easy to see that $\gamma(N)$  is $\Sigma_{\H^1}(X)$-negligible, so $\gamma(N) \in \Sigma_{\H^1}(X)$ and $\H^1(\gamma(N))=0$, because $\Sigma_{\H^1}(X)$ is complete for $\H^1$.
	
	\subsection{$\R$-trees} \label{ss:trees}
	An {\it $\R$-tree} is a metric space $(T,d_T)$ such that for any pair of points, there exists a unique arc $[x,y] \sbs T$ with endpoints $x$ and $y$, and this arc is the image of a geodesic path $\gamma_{xy}\colon [0,d_T(x,y)] \to T$. We will abuse terminology and refer to both $\gamma_{xy}$ and $[x,y]$ as geodesics. The \emph{interior} of the geodesic $\gamma_{xy}$   is  the image of $(0,d_T(x,y))$ by $\gamma_{xy}$.
	
	Let $T$ be an $\R$-tree. A subset $A \sbs T$ is {\it measurable} if $\gamma_{xy}^{-1}(A) \sbs \R$ is $\lambda$-measurable for all $x,y \in T$. The {\it length measure} is defined on measurable subsets by 
	\[\lambda_T(A)=\sup\Big\{\sum_i \lambda(\gamma_{x_i y_i}^{-1}(A)): \{[x_i,y_i]\}_i \text{ finite, pairwise disjoint sets of geodesics}\Big\}.\]
	It is clear that if $A \sbs T$ and $\lambda_T(A)<\infty$, then $A$ is totally bounded.
	
	The {\it leaf set} of $T$ is the subset $L(T)$ consisting of points $x\in T$ which are not in the interior of any geodesic. It is clear from the definition that $\lambda_T(L(T)) = 0$. If $\overline{T}$ denotes the completion of $T$, then $\overline{T}$ is also an $\R$-tree and $\overline{T}\setminus T \sbs L(\overline{T})$. Hence,
	\begin{equation} \label{eq:lambda=0}
		\lambda_{\overline{T}}\left(\overline{T}\setminus T\right) = 0.
	\end{equation}
	
	We introduce an elementary approximation-by-step-functions proposition that will be used later in the proof of our main result. Towards that end, we require the following lemma.
	
	\begin{lemma} \label{lem:1}
		Let $T$ be an $\R$-tree. For all measurable $A \sbs T$ and $x,y\in T$, we have $\lambda(\gamma_{xy}^{-1}(A))=\lambda_T(A \cap [x,y])$.
	\end{lemma}
	
	\begin{proof}
		Let $A \sbs T$ be measurable and $x,y\in T$.
		It is clear that \[\lambda(\gamma_{xy}^{-1}(A))=\lambda(\gamma_{xy}^{-1}(A \cap [x,y])) \le \lambda_T(A \cap [x,y]).\]
		Let $\{[x_i,y_i]\}_i$ be a finite set of pairwise disjoint geodesics. Note that for all $i$, uniqueness of geodesics implies that $[x,y]\cap[x_i,y_i]$ is a geodesic included in $[x,y]$ or is empty. Again using uniqueness of geodesics, we get that
		\[\sum_i\lambda(\gamma_{x_i y_i}^{-1}(A \cap [x,y]\cap[x_i,y_i]))=\sum_i\lambda(\gamma_{xy}^{-1}(A \cap [x,y]\cap[x_i,y_i])) \le \lambda(\gamma_{xy}^{-1}(A \cap [x,y])).\]
		Passing to the supremum over $\{[x_i,y_i]\}_i$ gives $\lambda_T(A \cap [x,y])\le \lambda(\gamma_{xy}^{-1}(A))$.
	\end{proof}
	
	\begin{proposition}[Step Function Approximation] \label{prop:stepapproximation}
		Let $T$ be an $\R$-tree and $A\sbs T$ measurable with $\lambda_T(A)<\infty$, let $f \in L_1(T,\lambda_T)$, and let $\eps>0$. Then
		\begin{itemize}
			\item there exist finitely many pairwise disjoint geodesics $\{[w_i,x_i]\}_i \sbs T$ such that $\lambda_T(A \Delta \cup_i [w_i,x_i]) < \eps$, and
			\item there exist finitely many scalars $\{c_i\}_i \sbs \R$ and pairwise disjoint geodesics $\{[y_i,z_i]\}_i \sbs T$ such that $\|f-\sum_i c_i \one_{[y_i,z_i]}\|_{L_1(T,\lambda_T)} < \eps$.
		\end{itemize}
	\end{proposition}
	
	\begin{proof}
		It suffices to prove the first item, as the second follows by approximating $f$ by a simple function. It follows from Lemma~\ref{lem:1} and the definition of $\lambda_T$ that for any $\eps>0$, there exist finitely many pairwise disjoint geodesics $\{[x_i,y_i]\}_i$ such that \[\left\|\one_A-\sum_i\one_{A \cap [x_i,y_i]}\right\|_{L_1(T,\lambda_T)}< \eps.\] Thus, it suffices to find the desired approximation for $\one_{A \cap [x,y]}$, for each $x,y \in T$. Such an approximation is possible by Lemma~\ref{lem:1} and standard results about approximations of functions in $L_1(\R,\lambda)$ by step functions. 
	\end{proof}
	
	In \cite[Corollary~3.3]{Godard2010PAMS}, Godard proved that for any $\R$-tree $T$, the free space $\F(T)$ is isometrically identified with $L_1(T,\lambda_T)$ via $\delta_y-\delta_x \mapsto \one_{[x,y]}$. We will use this well-known identification liberally in the sequel.
	
	\subsection{The Effros--Borel structure on closed subsets of a Polish space}
	
	For the proofs of both main results of this paper, we shall need to give a meaning to the sentence ``a map taking values in the set of closed subsets of a given Polish space is Borel". This will be done by using the so-called Effros--Borel structure. Let us recall the definitions. 
	
	A \emph{Polish space} (resp.\ \emph{topology}) is a separable completely metrizable topological space (resp.\ topology). A set $X$ equipped with a $\sigma$-algebra is called a \emph{standard Borel space} if the $\sigma$-algebra is generated by a Polish topology on $X$. A subset of such a  standard Borel space $X$ is called \emph{Borel} if it is an element of the corresponding $\sigma$-algebra.
	
	Let $X$ be a Polish space. We denote $Cl(X)$ the set of all closed subsets of $X$. Then, $Cl(X)$ can be equipped with its \emph{Effros--Borel structure}, defined as the $\sigma$-algebra generated by the sets $\{F \in Cl(X): F\cap U\neq \emptyset\}$,  where $U$ varies over the open subsets of $X$. Equipped with this $\sigma$-algebra, ${Cl}(X)$ is a standard Borel space.\footnote{In the literature, ${Cl}(X)$ is often denoted by $\cal F(X)$. This notation is however already reserved for the Lipschitz-free space in this paper.}
	
	We include, for future reference the following classical and easy to prove result.
	
	\begin{proposition}\emph{(cf.\ \cite[Exercises 4.21 and 12.7]{KechrisBook1995}).} Let $K$ be a compact metric space. Then the Effros--Borel structure is also generated by the sets $\{F \in Cl(K)\colon F \sbs V\}$, where $V$ runs over the open subsets of $K$. It also coincides with the Borel $\sigma$-algebra associated with the Hausdorff metric on $Cl(K)$.\label{p:EffrosBorel_on compact} 
	\end{proposition}
	
	This construction can be adapted for the set of closed linear subspaces of a given separable Banach space, as it was initiated in the work of  Bossard in \cite{Bossard2002}. So, let $E$ be a separable Banach space.  We equip ${Cl}(E)$ with its corresponding Effros--Borel structure. We denote 
	\[\mathcal S(E)=\{F\in {Cl}(E) : F\ \text{is a linear subspace of}\ E\},
	\]
	considered as a subset of  ${Cl}(E)$. Then $\mathcal S(E)$ is a Borel subset of ${Cl}(E)$ (\cite[Proposition 2.2]{Bossard2002}) and therefore a standard Borel space. In the sequel, $\mathcal S(E)$ will always be equipped with this Borel structure.

	\section{\texorpdfstring{Metric spaces with locally finite $\H^1$-measure}{Metric spaces with locally finite -measure}}
	
	\subsection{Parametrization of compact, connected spaces with finite $\H^1$-measure}
	
	A version of the following theorem was first proved by Ważewski \cite{Waz}. It gives a characterization of compact, connected metric spaces with finite $\H^1$-measure in terms of a canonical parametrization. This parametrization will be essential in our study of Lipschitz-free spaces over such metric spaces.
	
	\begin{theorem}[{\cite[Theorem~4.4]{Golab}}] \label{thm:Waz}
		Let $X$ be a compact, connected metric space with $\H^1(X)<\infty$. Then there exists a map $\gamma\colon [0,2\H^1(X)] \to X$ with the following properties:
		\begin{enumerate}
			\item $\gamma$ is surjective. 
			\item $\gamma$ is 1-Lipschitz with unit speed.
			\item For $\H^1$-a.e. $x\in X$, we have $|\gamma^{-1}(\{x\})| = 2$.
		\end{enumerate}
	\end{theorem}
	
	We need a corollary giving a slightly stronger version of the previous theorem that allows us to start the map at an arbitrary point. We believe that it can be achieved with the same construction yielding Theorem~\ref{thm:Waz}, but it also follows from a simple doubling argument that we present next.
	
	\begin{corollary} \label{cor:Waz}
		Let $X$ be a compact, connected metric space with $\H^1(X)<\infty$. Then for any $x_0 \in X$, there exists a map $\gamma\colon [0,4\H^1(X)] \to X$ with the following properties:
		\begin{enumerate}
			\item $\gamma$ is surjective.
			\item $\gamma(0) = x_0$.
			\item $\gamma$ is 1-Lipschitz with unit speed.
			\item For $\H^1$-a.e. $x\in X$, we have $|\gamma^{-1}(\{x\})| = 4$.
		\end{enumerate}
	\end{corollary}
	
	\begin{proof}
		Let $\gamma\colon [0,2\H^1(X)] \to X$ be a map satisfying the conclusion of Theorem~\ref{thm:Waz}. Let $\bar{\gamma}\colon [2\H^1(X),4\H^1(X)]\to X$ be the reverse of $\gamma$, meaning $\bar{\gamma}(t) = \gamma(4\H^1(X)-t)$ for $t\in [2\H^1(X),4\H^1(X)]$. Then we form the concatenation
		$\gamma*\bar{\gamma}: [0,4\H^1(X)] \to X$ defined by $\gamma*\bar{\gamma}\big|_{[0,2\H^1(X)]} = \gamma$ and $\gamma*\bar{\gamma}\big|_{[2\H^1(X),4\H^1(X)]} = \bar{\gamma}$. Then $\gamma*\bar{\gamma}$ satisfies the desired properties, except that $\gamma*\bar{\gamma}(0)$ equals $\gamma(0)$, while we need it to equal $x_0$. But note that $\gamma*\bar{\gamma}$ is actually a map on the circle of arclength $4\H^1(X)$ because $\gamma*\bar{\gamma}(0) = \gamma(0) = \gamma*\bar{\gamma}(4\H^1(X))$. Then, since $\gamma$ is surjective, by isometrically rotating the domain, we may assume that the map sends 0 to $x_0$. 
	\end{proof}
	
	As we shall need to deal with maps $\gamma$ as in Corollary \ref{cor:Waz} often, the next definition introduces a terminology for such maps. 
	
	\begin{definition}
		In recognition of Ważewski's work \cite{Waz}, we call any $\gamma\colon [0,4\H^1(X)] \to X$ satisfying the conclusion of Corollary~\ref{cor:Waz} a {\it W-map}.
	\end{definition}

	\subsection{Tree Parametrization}
	
	The following proposition, relying heavily on Kirchheim's theorems, is the key technical result allowing us to prove our main theorem. It says that a W-map admits a partial inverse that is infinitesimally 1-Lipschitz almost everywhere.
	
	\begin{proposition} \label{prop:gamma^(-1)}
		Let $(X,d)$ be a compact, connected metric space with $0<\H^1(X)<\infty$. Set $I := [0,4\H^1(X)]$, and let $\gamma\colon I \to X$ be a W-map. Then the following statements are true:
		\begin{enumerate}
			\item $M\! D(\gamma,t) = 1$ for $\lambda$-a.e. $t\in I$.
			\item There exists a Borel subset $A \sbs I$ such that
			\begin{enumerate}
				\item $\gamma\big|_A$ is injective,
				\item $\gamma(A)$ is Borel and has full $\H^1$-measure in $X$, and
				\item for every Borel subset $B \sbs \gamma(A)$, there exists a Borel set $B' \sbs B$ with full $\H^1$-measure in $B$ such that for every $x\in B'$, there exists a sequence $(x_n)_{n=0}^\infty\sbs B'\setminus\{x\}$ with $\lim_n x_n= x$ and $\lim_{n\to\infty} \frac{|\psi(x)-\psi(x_n)|}{d(x,x_n)} = 1$, where $\psi: \gamma(A) \to A$ denotes the inverse of $\gamma\big|_A$.
			\end{enumerate}
		\end{enumerate}
	\end{proposition}
	
	\begin{proof}
		Note that $\Lip(\gamma) \leq 1$ clearly implies that $0\leq M\! D(\gamma,t) \leq 1$ for every $t \in I$ for which $M\! D(\gamma,t)$ exists, which is $\lambda$-a.e.\ $t\in I$ by Kirchheim Differentiation (Theorem~\ref{eq:Kirchheimdiff}). By Corollary~\ref{cor:Waz}(3) and the Area Formula (Theorem~\ref{thm:areaformula}), this implies
		\begin{equation*}
			4\H^1(X) \geq \int_I M\! D(\gamma,t)\, d\lambda(t) = \int_{\gamma(I)} 4\, d\H^1(y) = 4\H^1(X).
		\end{equation*}
		Hence it must happen that $M\! D(\gamma,t) = 1$ for $\lambda$-a.e. $t \in I$. This completes the proof of (1).
		
		For (2), we first use Corollary~\ref{cor:Waz}(3) and the results in Section \ref{subsec:H_1} to pick $E \in \cal B(X)$, with full $\cal H^1$-measure and such that for all $x\in E$, $|\gamma^{-1}(\{x\})| = 4$. Then we claim that there exists a Borel set $A \sbs I$ such that $\gamma\big|_A$ is injective and $\gamma(A)=E$. We first show that $\phi_1: x\mapsto \gamma^{-1}(\{x\})$ is Borel from $X$ to ${\cal Cl}(I)$ equipped with its Effros--Borel structure. Indeed, for a fixed closed subset $F$ of $I$, we have that $\gamma^{-1}(\{x\})\cap F \neq \emptyset$ if and only if $x\in \gamma(F)$. Since $F$ is compact, so is $\gamma(F)$. Since $I$ is compact, this shows, by Proposition \ref{p:EffrosBorel_on compact}, that $\phi_1$ is Borel. Recall (Proposition \ref{p:EffrosBorel_on compact}) that $I$ being compact also implies that the Effros--Borel structure of ${\cal Cl}(I)$ coincide with the Borel $\sigma$-algebra associated with the Hausdorff metric, for which the map $\phi_2\colon F\mapsto \min(F)$ is clearly continuous from ${\cal Cl}(I)\setminus \{\emptyset\}$ to $I$. So $\psi=\phi_2 \circ \phi_1\colon E\to I$ is Borel and injective. It finally follows from a theorem of Lusin and Souslin (see \cite[Theorem 15.1]{KechrisBook1995}) that $A=\psi(E)$ is Borel, and of course, $\gamma(A)=E$ and $\gamma\big|_A=\psi^{-1}$ is injective. This shows (a) and (b).
		
		By part (1), we may further assume that $M\! D(\gamma,t) = 1$ for every $t\in A$. Let $B \sbs \gamma(A)$ be Borel. Note that $\psi(B)$ is a Borel subset of $A$, due to the previously mentioned theorem of Lusin and Souslin. So there exists a sequence $(A_j)_j$ of compact subsets of $\psi(B)$ such that $\cup_j A_j$ has full $\lambda$-measure in $\psi(B)$. Since $\gamma$ is Lipschitz, $K_j=\gamma(A_j)$ is compact, $\cup_jK_j$ has full $\cal H^1$-measure in $B$, and $\gamma\big|_{A_j}$ is a homeomorphism from $A_j$ onto $K_j$, with inverse $\psi\big|_{K_j}$. Using repeatedly the Cantor--Bendixson derivation, we take off the $K_j$'s $\H^1$-null sets, so we may assume that each $K_j$ is perfect (we are using here that $\H^1(X)>0$). We set $B' := \cup_j K_j$. Now let $x \in B'$, so that $x \in K_j$ for some $j\in\N$. Since $K_j$ is perfect, we may find a sequence $(x_n)_n\sbs K_j\setminus\{x\}$ such that $x_n\to x$. Then, since $\psi\big|_{K_j}$ is continuous, $(\psi(x_n))_n \sbs A\setminus\{\psi(x)\}$ and $ \psi(x_n) \to \psi(x) \in A$, and hence we have
		\begin{align*}
			1 = M\! D(\gamma,\psi(x)) = \lim_{n\to\infty} \frac{d(\gamma(\psi(x)),\gamma(\psi(x_n)))}{|\psi(x)-\psi(x_n)|} = \lim_{n\to\infty} \frac{d(x,x_n)}{|\psi(x)-\psi(x_n)|},
		\end{align*}
		implying the desired result.
	\end{proof}
	
	\begin{lemma} \label{lem:W-mapgeodesics}
		Let $X$ be a compact, connected metric space with $0<\H^1(X)<\infty$. Set $I := [0,4\H^1(X)]$, and let $\gamma\colon I \to X$ be a W-map. Then for every geodesic path $g: [a,b]\to X$ and every Lipschitz map $f\colon X\to\R$, we have that $\Lip(f \circ g) \leq \Lip(f\circ\gamma)$.
	\end{lemma}
	
	\begin{proof}
		Let $f: X \to \R$ be Lipschitz and $g: [a,b]\to X$ a geodesic path. The composition $f\circ g$ is a Lipschitz function from an interval to $\R$, and hence $\Lip(f\circ g) = \|(f\circ g)'\|_\infty$ by the Lebesgue Fundamental Theorem of Calculus. Therefore, we need to show that $\|(f\circ g)'\|_\infty \leq \Lip(f\circ\gamma)$.
		
		Let $A \sbs I$ be a Borel subset satisfying the conclusion of Proposition~\ref{prop:gamma^(-1)}(2). Let $A' \sbs [a,b]$ be Borel and of full $\lambda$-measure in $[a,b]$ such that $(f\circ g)'(t)$ exists for every $t\in A'$. We recall that $\gamma(A)$ is Borel in $X$. Since $g$ is an isometry from $[a,b]$ onto its compact image, it is clear that $g(A')$ is also Borel in $X$. Applying Proposition~\ref{prop:gamma^(-1)}(2)(c) to the Borel subset $B := g(A')\cap\gamma(A)$ of $X$, we get that there exists a Borel set $B' \sbs B$ with full $\H^1$-measure in $B$ such that for every $x\in B'$, there exists a sequence $(x_n)_n\sbs B'\setminus\{x\}$ with $x_n\to x$ and $\lim_{n\to\infty} \tfrac{|\psi(x)-\psi(x_n)|}{d(x,x_n)} = 1$, where $\psi\colon \gamma(A) \to A$ still denotes the inverse of $\gamma\big|_A$.
		
		Now consider the Borel set $g^{-1}(B') \sbs A' \cap g^{-1}(\gamma(A))$. This set has full $\lambda$-measure in $[a,b]$ since $g$ is an isometric embedding and since\\
		(i) $B'$ has full $\H^1$-measure in $g(A')\cap \gamma(A)$,\\
		(ii) $A'$ has full $\lambda$-measure in $[a,b]$, and\\ (iii) $\gamma(A)$ has full $\H^1$-measure in $X$.
		
		Let $t \in g^{-1}(B')$. Then $g(t) \in B'$, and hence, by definition of $B'$, we can find a sequence $(x_n)_n\sbs B'\setminus\{g(t)\}$ with $x_n\to g(t)$ and
		\begin{equation} \label{eq:MD(gamma^(-1))1}
			\lim_{n\to\infty} \dfrac{|\psi(g(t))-\psi(x_n)|}{d(g(t),x_n)} = 1.
		\end{equation}
		Then, since $g$ is an isometric embedding, there is a sequence $(t_n)_n \sbs A'\setminus\{t\}$ such that $t_n \to t \in A'$ and $g(t_n) = x_n$. Hence, by definition of $A'$, we have that
		\begin{equation} \label{eq:f(g)'}
			(f \circ g)'(t) = \lim_{n\to\infty} \frac{f(g(t))-f(g(t_n))}{t-t_n}.
		\end{equation}
		Also, plugging $g(t_n) = x_n$ into \eqref{eq:MD(gamma^(-1))1} and using that $g$ is isometric gives us
		\begin{equation} \label{eq:MD(gamma^(-1))2}
			\lim_{n\to\infty} \dfrac{|\psi(g(t))-\psi(g(t_n))|}{|t-t_n|} = 1.
		\end{equation}
		Combining these last two equations yields
		\begin{align*}
			|(f \circ g)'(t)| &\overset{\eqref{eq:f(g)'}}{=} \lim_{n\to\infty} \frac{|f(g(t))-f(g(t_n))|}{|t-t_n|} \\
			&= \lim_{n\to\infty} \frac{|(f\circ\gamma)(\psi(g(t)))-(f\circ\gamma)(\psi(g(t_n)))|}{|\psi(g(t))-\psi(g(t_n))|} \frac{|\psi(g(t))-\psi(g(t_n))|}{|t-t_n|} \\
			&\leq \limsup_{n\to\infty} \Lip(f\circ\gamma) \frac{|\psi(g(t))-\psi(g(t_n))|}{|t-t_n|} \\
			&\overset{\eqref{eq:MD(gamma^(-1))2}}{=} \Lip(f\circ\gamma).
		\end{align*}
		Since $g^{-1}(B')$ has full $\lambda$-measure in $[a,b]$ and since $t \in g^{-1}(B')$ was arbitrary, this shows that $\|(f \circ g)'\|_\infty \leq \Lip(f\circ\gamma)$.
	\end{proof}
	
	We now arrive at the main result of this section. We say that an $\R$-tree $T$ has {\it locally finite $\lambda_T$-measure} if $\lambda_T(K)<\infty$ for all compact $K\sbs T$.
	
	\begin{theorem}[Tree Parametrization] \label{thm:isometric}
		Let $X$ be a pointed, proper, geodesic metric space with locally finite $\H^1$-measure. Then there is a pointed proper $\R$-tree $T$ with locally finite $\lambda_T$-measure and a proper $0$-preserving Lipschitz map $\gamma: T \to X$ such that the pullback map $f \mapsto f\circ\gamma\colon \Lip_0(X) \to \Lip_0(T)$ is an isometric embedding.
	\end{theorem}
	
	\begin{proof}
		Let $0$ denote the basepoint of $X$. Let $(r_n)_{n=0}^\infty$ be a sequence of positive radii strictly increasing to $\infty$. Let $A_0$ be the closed ball of radius $r_0$ centered at $0$. For $n \geq 1$, let $A_n$ be the closed annulus of inner radius $r_{n-1}$ and outer radius $r_n$ centered at $0$. Define $A_{\leq n} := \cup_{k\leq n}A_k$. We will recursively construct an increasing sequence of compact $\R$-trees $(T_n)_{n\geq 0}$ and  1-Lipschitz surjective maps $(\gamma_n\colon T_n \to A_{\leq n})_{n\geq 0}$ such that for all $n\geq 0$, 
		\begin{enumerate}
			\item   ${\gamma_{n+1}}$ extends $\gamma_n$,
			\item $\lambda_{T_n}(T_n) = 4\H^1(A_{\leq n})$,
			\item $\gamma_m^{-1}(A_n) \sbs T_{n}$ for all $m \geq n$, and
			\item for all geodesic paths $g_X\colon [a,b] \to A_{n}$, there exists a geodesic path $g_T\colon [\alpha,\beta] \to T_{n}$ such that $\gamma_{n}\circ g_T\colon [\alpha,\beta] \to X$ is a W-map and $\gamma_n(g_T([\alpha,\beta])) \sps g_X([a,b])$.
		\end{enumerate}
		
		The base case $n=0$ will follow along the same reasoning as the inductive step and we omit the details. Let $n \geq 1$ and suppose that $\gamma_{n-1}\colon T_{n-1} \to A_{\leq n-1}$ has been constructed satisfying (1)--(3) above. Let $C$ be a connected component of $A_n$. We will define the map $\gamma_n$ individually on each such $C$. Since the whole space $X$ is geodesic, the Intermediate Value Theorem guarantees the existence of a point $x_C \in C \cap A_{n-1}$. Since $C$ is compact, connected, and has finite $\H^1$-measure, Corollary~\ref{cor:Waz} yields a W-map $\gamma_C\colon [0,4\H^1(C)] \to C$ such that $\gamma_C(0) = x_C$. Then we attach a branch $B_C$ with length $4\H^1(C)$ to $T_{n-1}$ at any chosen point $t_C \in \gamma_{n-1}^{-1}(\{x_C\})$. We isometrically identify $[0,4\H^1(C)]$ with $B_C$ so that $0$ gets identified with $t_C$, and thus get a glued  1-Lipschitz map $\gamma_{n-1} \cup \gamma_C\colon T_{n-1} \cup B_C \to A_{\leq n-1} \cup C$. Repeating this over all the connected components $C$ of $A_n$ and taking the completion yields the desired $\R$-tree $T_n$ and 1-Lipschitz map $\gamma_n\colon T_n \to A_{\leq n}$. Item (1) holds by construction and by \eqref{eq:lambda=0}. Item (2) obviously holds by construction. Compactness of $T_n$ holds since it is complete by definition, and since is totally bounded because $\lambda_{T_n}(T_n)<\infty$. Item (3) is also easy to see, but we will provide a short explanation. Let $g_X\colon [a,b] \to A_{n}$ be a geodesic path. Since $g_X([a,b]) \sbs A_n$, it is contained in a connected component $C$ of $A_n$. It follows from the construction that there is a geodesic path $g_T\colon [0,4\H^1(C)] \to T_n$ such that $\gamma_n\circ g_T\colon [0,4\H^1(C)] \to X$ is a W-map and $\gamma_n(g_T([0,4\H^1(C)])) = C \sps g_X([a,b])$.
		
		Now, choose two sequences of radii $(r_n^{(1)})_{n\geq 0}$ and $(r_n^{(2)})_{n\geq 0}$ that intertwine with 1-separation, meaning
		$$0< r_0^{(1)} < r_0^{(2)} -1 < r_1^{(1)} - 2 < r_1^{(2)} - 3 < r_2^{(1)} - 4 < \dots.$$
		For $i\in \{1,2\}$, apply the construction of the previous paragraph using the sequence $(r_n^{(i)})_{n=0}^\infty$ to obtain an $\R$-tree $T^{(i)} := \cup_{n\geq 0} T^{(i)}_n$ and a 1-Lipschitz map \[\gamma^{(i)} := \cup_{n\geq 0}\gamma^{(i)}_n\colon  T^{(i)} \to \cup_{n\geq 0}A^{(i)}_n = X.\] Items (1) and (2) imply that $T^{(i)}$ is proper with locally finite $\lambda_{T^{(i)}}$-measure and that $\gamma^{(i)}$ is proper (and $\gamma^{(i)}$ is 1-Lipschitz since each $\gamma_n^{(i)}$ is 1-Lipschitz). We form another $\R$-tree $T$ and 1-Lipschitz map $\gamma\colon T \to X$ by wedging $\gamma^{(1)}\colon T^{(1)} \to X$ and $\gamma^{(2)}\colon  T^{(2)} \to X$ together along base points $t_0^{(1)} \in T^{(1)}$ and $t_0^{(2)} \in T^{(2)}$ with $\gamma^{(i)}(t_0^{(i)}) = 0$: $t_0^{(1)}$ and $t_0^{(2)}$ are identified with the base point $t_0$ of $T$.  It is clear that $T$ is a proper $\R$-tree with locally finite $\lambda_T$-measure. In order to show that the pullback under $\gamma$ is an isometric embedding from $\Lip_0(X)$ into $\Lip_0(T)$, it suffices to show that $\Lip(f\circ g_X)\leq \Lip(f\circ\gamma)$ for every Lipschitz map $f: X\to\R$ and every  geodesic path $g_X\colon [0,1] \to X$. Indeed, once this has been shown, since $X$ is geodesic and $\gamma$ is 1-Lipschitz, we have
		\[\Lip(f)=\sup_{\substack{g_X\colon [0,1]\to X\\ g_X\text{ geodesic}}}\Lip(f\circ g_X)\leq \Lip(f\circ \gamma)\leq \Lip(f).\]
		
		Let $g_X\colon [0,1] \to X$ be a geodesic   and let us show that $  \Lip(f\circ g_X)\leq \Lip(f\circ\gamma) $. Since $g_X$ is an isometry  and since the radii $(r_n^{(1)})_{n\geq 0},(r_n^{(2)})_{n\geq 0}$ intertwine with 1-separation, we must have that $g_X([0,1]) \sbs A_n^{(i)}$ for some $i\in\{1,2\}$ and $n\geq 0$. Then by (3), there exists a geodesic path $g_T\colon [\alpha,\beta] \to A_{n}^{(i)}$ such that $\gamma_{n}^{(i)}\circ g_T\colon [\alpha,\beta] \to X$ is a W-map and $\gamma_n^{(i)}(g_T([\alpha,\beta])) \sps g_X([0,1])$. Then by Lemma~\ref{lem:W-mapgeodesics}, we have that \[\Lip(f\circ g_X)\le \Lip(f \circ \gamma_n^{(i)} \circ g_T)=\Lip(f\circ\gamma_n^{(i)})\leq \Lip(f\circ \gamma).\qedhere\] 
	\end{proof}
	
	\subsection{The topology of local convergence in measure}
	
	The aim of this subsection is to prove the following proposition.
	
	\begin{proposition}[Vanishing Supports] \label{prop:vanishingsupport}
		Let $(X,d,0)$ be a pointed metric space and $\mu\in\F(X)$. Suppose that for every $N<\infty$, there exist $\mu_1,\mu_2\in\F(X)$ such that
		\begin{itemize}
			\item $\mu_1 \in \spn\{\delta_{y_i}-\delta_{x_i}\}_i$ for some finite subset $\{x_i,y_i\}_i \sbs X$ with $\sum_i d(x_i,y_i) \leq \frac{1}{N}$,
			\item $\mu_2\in\F((X\setminus B_{N}(0))\cup\{0\})$, where $B_N(0)=\{x\in X :d(x,0)\le N\}$, and
			\item $\|\mu-\mu_1-\mu_2\| \leq \frac{1}{N}$.
		\end{itemize}
		Then $\mu = 0$.
	\end{proposition}
	
	Before proving the proposition, let us recall the notion of quotient metric spaces. Given an equivalence relation $\sim$ on a metric space $(X,d)$, let $X/{\sim}$ denote the set of equivalence classes equipped with the quotient semimetric $$d([p],[q]) = \inf\left\{\sum_{i=1}^n d(x_i,y_i): x_1 \sim p, y_i\sim x_{i+1}, y_n\sim q\right\}.$$
	If this semimetric is also a metric (i.e., if it is positive definite), then the equivalence relation $\sim$ is said to be {\it closed}. In the following proof, all the equivalence relations we consider are closed. If $F \sbs X$ is a closed subset, we write $X/F$ for $X/{\sim}$, where $\sim$ is the (closed) equivalence relation generated by $p\sim q$ if $p,q\in F$. If $X$ has basepoint $0$, we equip $X/{\sim}$ with basepoint $[0]$, which we will also write as $0$.
	
	\begin{proof}[Proof of Proposition~\ref{prop:vanishingsupport}]
		Let $\eps>0$ be arbitrary. Choose $\mu_f\in\F(X)$ with finite support of cardinality at least 2, such that $\|\mu-\mu_f\|< \eps$. Let \[r := \min_{p\neq q\in\supp(\mu_f)} d(p,q)\ \text{ and }\   R:= \max_{p\in\supp(\mu_f)} d(p,0).\] Let $N<\infty$ such that $N-\tfrac{1}{N} \geq 2R$ and $N > \tfrac{1}{r}$. By hypothesis, we can choose $\mu_1,\mu_2\in\F(X)$ such that
		\begin{itemize}
			\item $\mu_1 \in \spn\{\delta_{y_i}-\delta_{x_i}\}_i$ for some finite subset $\{x_i,y_i\}_i \sbs X$ with $\sum_i d(x_i,y_i) \leq \frac{1}{N}$,
			\item $\mu_2\in\F((X\setminus B_{N}(0))\cup\{0\})$, and
			\item $\|\mu-\mu_1-\mu_2\|_{\F(X)} \leq \frac{1}{N}$.
		\end{itemize}
		
		Let $\sim$ be the equivalence relation on $X$ generated by $x_i \sim y_i$ for all $i$. Let $X_1$ denote the quotient space $X/\hspace{-.1cm}\sim$, and denote by $d_1$ the quotient metric. It is easy to see that \begin{equation}\label{Eq.8oct.26.1}d_1([p],[q]) \geq d(p,q) - \sum_i d(x_i,y_i) \geq d(p,q) - \tfrac{1}{N}\end{equation}
		for all $[p],[q]\in X_1$. Obviously, the quotient map $f_1:X \to X_1$ is 1-Lipschitz, and \eqref{Eq.8oct.26.1} implies that \[d_1([p],[q]) \geq   \left(1-\frac{1}{Nr}\right)d(p,q), \ \ \text{if}\  d(p,q)\geq r.\]
		So, the restriction of  $f$ to any $r$-separated subset of $X$ is  $(1-\tfrac{1}{Nr})$-coLipschitz. In particular, this is the case for the subset  $\supp(\mu_f)$ of $X$. Therefore \[\|\widehat{f_1}(\mu_f)\|_{\F(X_1)} \geq \left(1-\frac{1}{Nr}\right)\|\mu_f\|_{\F(X)}.\] Also, we clearly have that $\widehat{f_1}(\mu_1) = 0 \in \F(X_1)$, and \eqref{Eq.8oct.26.1}  implies that $\widehat{f_1}(\mu_2) \in \F(F_1)$, where \[F_1=\{x\in X_1: d_1(x,0) \geq N-\tfrac{1}{N}\} \cup\{0\}.\]
		
		Now let $X_{12}$ denote the quotient space $X_1/F_1$ and $d_{12}$ the quotient metric. Let $f_2\colon X_1 \to X_{12}$ denote the quotient map and $f_{12}\colon X \to X_{12}$ the composition $f_{12} := f_2\circ f_1$. We clearly have that $\widehat{f_{12}}(\mu_2) = 0 \in \F(X_{12})$ and that $f_{2}$ restricted to \[\left\{f_1(x)\in X_1:  x \in B_{\tfrac{1}{2}(N-\tfrac{1}{N})}(0)\right\}\] is an isometric embedding into $X_{12}$. In particular, since $\mu_f \in \F(B_R(0))$, since $f_1$ is 1-Lipschitz, and since we assumed that $R \leq \tfrac{1}{2}(N-\tfrac{1}{N})$, the last sentence implies that $\|\widehat{f_{12}}(\mu_f)\|_{\F(X_{12})} = \|\widehat{f_{1}}(\mu_f)\|_{\F(X_{1})}$. Putting this all together, we get that
		\begin{align*}
			\|\mu_f - \mu_1 - \mu_2\|_{\F(X)} &\geq \|\widehat{f_1}(\mu_f - \mu_1 - \mu_2)\|_{\F(X_1)} \\
			&= \|\widehat{f_1}(\mu_f - \mu_2)\|_{\F(X_1)} \\
			&\geq \|\widehat{f_2}\widehat{f_1}(\mu_f - \mu_2)\|_{\F(X_{12})} \\
			&= \|\widehat{f_{12}}(\mu_f)\|_{\F(X_{12})} \\
			&= \|\widehat{f_{1}}(\mu_f)\|_{\F(X_{1})} \\
			&\geq \left(1-\frac{1}{Nr}\right)\|\mu_f\|_{\F(X)}.
		\end{align*}
		Combining this with the initial bounds $\|\mu-\mu_f\|_{\F(X)}<\eps$ and $\|\mu-\mu_1-\mu_2\|_{\F(X)} \leq \frac{1}{N}$, we get that
		\begin{equation*}
			\tfrac{1}{N}+\eps \geq \left(1-\frac{1}{Nr}\right)(\|\mu\|_{\F(X)}-\eps).
		\end{equation*}
		Taking $N\to\infty$ yields $\eps \geq \|\mu\|_{\F(X)}-\eps$, and then $\eps\to 0$ results in $\mu=0$.
	\end{proof}

	\begin{definition}
		Given a measure $\mu$, we say that a sequence of functions $(\phi_n)_n$ in $L_1(\mu)$ converges {\it locally in measure} to $\phi \in L_1(\mu)$ if $(\phi_n\one_A)_n$ converges in measure to $\phi\one_A$ for every $\mu$-measurable set $A$ with $\mu(A)<\infty$.   
	\end{definition}
	
	We can now apply Proposition \ref{prop:vanishingsupport} to obtain the main result of this section. The statement and its proof use Godard's identification of $\F(T)$ with $L_1(T,\lambda_T)$ described in \S\ref{ss:trees}.
	
	\begin{theorem} \label{thm:kernel}
		Let $(T,d_T,0)$ be a proper $\R$-tree with locally finite $\lambda_T$-measure and $(X,d_X,0)$ a proper pointed metric space. Let $\gamma: T \to X$ be a base-point-preserving, proper Lipschitz map and $\widehat{\gamma}: L_1(T,\lambda_T) = \F(T) \to \F(X)$ its linearization. Then $\ker(\widehat{\gamma})$ is closed with respect to the topology of local convergence in measure.
	\end{theorem}
	
	\begin{proof}
		By rescaling $d_X$, observing that every hypothesis of the theorem is preserved under this rescaling, and observing that $\ker(\widehat{\gamma})$ is invariant under this rescaling, we may assume that $\Lip(\gamma) \leq 1$. It is enough to show that for all $\mu\in\F(X)$ and all sequences $(\phi_n)_n$ in $L_1(T,\lambda_T)$, if $(\phi_n)_n$ converges locally in measure to 0 and $\widehat{\gamma}(\phi_n) = \mu$ for all $n$, then $\mu=0$. Let $\mu\in \F(X)$ and let $(\phi_n)_n$ be such a sequence in $L_1(T,\lambda_T)$. We will use Proposition~\ref{prop:vanishingsupport} to deduce that $\mu=0$.
		
		Let $N<\infty$ be arbitrary. Choose a radius $R<\infty$ such that
		\begin{equation} \label{eq:farsupport}
			\gamma(T\setminus B_R(0)) \sbs X\setminus B_N(0)
		\end{equation}
		Such an $R<\infty$ exists since $\gamma$ is a proper map and $X$ is a proper space. Since $T$ is proper with locally finite $\lambda_T$-measure, we have that $\lambda_T(B_R(0)) < \infty$. Thus, since $(\phi_n)_n$ converges to 0 locally in measure, $(\phi_n\one_{B_R(0)})_n$ converges to 0 in measure. Egoroff's theorem then implies that we can find a measurable subset $A\sbs B_R(0)$ and an index $n$ such that
		\begin{equation} \label{eq:smallsupport1}
			\lambda_T(A) < \tfrac{1}{2N}
		\end{equation}
		and
		\begin{equation} \label{eq:Egoroff}
			\|\phi_n\one_{B_R(0)\setminus A}\|_{L_1(T,\lambda_T)} < \tfrac{1}{2N}.
		\end{equation}
		
		Applying Proposition~\ref{prop:stepapproximation} separately to $\phi_n$ and to $\one_A$, and using the fact that intersection of geodesics in $\R$-trees are geodesic (due to uniqueness of geodesics), we get that there exist finitely many scalars $\{c_i\}_i \sbs \R$ and pairwise disjoint geodesics $\{[x_i,y_i]\}_i \sbs T$ such that
		\begin{equation} \label{eq:stepapprox}
			\left\|\phi_n\one_A -  \sum_i c_i \one_{[x_i,y_i]}\right\|_{L_1(T,\lambda_T)} < \tfrac{1}{2N} 
		\end{equation}
		and \[\sum_i d_T(x_i,y_i) = \lambda_T\left(\bigcup_i [x_i,y_i]\right) < \lambda_T(A) + \tfrac{1}{2N} \overset{\eqref{eq:smallsupport1}}{<} \tfrac{1}{N}.\] Note that, since $\gamma$ is 1-Lipschitz, this implies
		\begin{equation} \label{eq:smallsupport2}
			\sum_i d_X(\gamma(x_i),\gamma(y_i)) < \tfrac{1}{N}.
		\end{equation}
		
		Set
		\begin{equation} \label{eq:mu1}
			\mu_1 := \widehat{\gamma}\left( \sum_i c_i \one_{[x_i,y_i]}\right) = \sum_i c_i(\delta_{\gamma(y_i)}-\delta_{\gamma(x_i)})
		\end{equation}
		and
		\begin{equation} \label{eq:mu2}
			\mu_2 := \widehat{\gamma}(\phi_n\one_{T\setminus B_R(0)}) \overset{\eqref{eq:farsupport}}{\in} \F((X\setminus B_N(0))\cup\{0\}).
		\end{equation}
		Then we have
		\begin{align} \label{eq:approx}
			\nonumber    \|\mu - \mu_1 - \mu_2\|_{\F(X)} &= \left\|\widehat{\gamma}\left(\phi_n -  \sum_i c_i \one_{[x_i,y_i]} - \phi_n\one_{T\setminus B_R(0)}\right)\right\|_{\F(X)} \\
			\nonumber    &\leq \left\|\phi_n -  \sum_i c_i \one_{[x_i,y_i]} - \phi_n\one_{T\setminus B_R(0)}\right\|_{L_1(T,\lambda_T)} \\
			\nonumber    &\leq \left\|\phi_n\one_{B_R(0)\setminus A}  \right\|_{L_1(T,\lambda_T)} + \left\|\phi_n\one_{A} -  \sum_i c_i \one_{[x_i,y_i]}\right\|_{L_1(T,\lambda_T)} \\
			&\overset{\eqref{eq:Egoroff},\eqref{eq:stepapprox}}{<} \tfrac{1}{2N} + \tfrac{1}{2N} = \tfrac{1}{N}.
		\end{align}
		Since $N<\infty$ was arbitrary, it follows from \eqref{eq:smallsupport2} and \eqref{eq:mu1}, from \eqref{eq:mu2}, and from \eqref{eq:approx}, that the hypotheses of Proposition~\ref{prop:vanishingsupport} are satisfied. Therefore, by that proposition, $\mu=0$.
	\end{proof}
	
	\subsection{Main Results}
	
	First, we need to recall classical results due to   Godefroy.
	\begin{theorem}[{\cite[Théorème~3]{Godefroy}}] \label{thm:L-embedded}
		Suppose $\mu$ is a finite measure and $Z$ a closed subspace of $L_1(\mu)$ such that $Z$ is closed with respect to the topology of convergence in measure. Then $L_1(\mu)/Z$ is $L$-embedded in its bidual.
	\end{theorem}
	
	The following is an immediate corollary. We include its simple proof. 
	
	\begin{corollary} \label{cor:localconvergence}
		Suppose $\mu$ is a $\sigma$-finite measure on a measurable space $(\Omega,\Sigma)$ and $Z$ a closed subspace of  $L_1(\mu)$ such that $Z$ is closed with respect to the topology of local convergence in measure.  Then $L_1(\mu)/Z$ is $L$-embedded in its bidual.
	\end{corollary}
	
	\begin{proof} There is a increasing sequence $(A_n)_{n=1}^\infty$ in $\Sigma$ such that for all $n \in \N$, $\mu(A_n)<\infty$ and $\Omega=\bigcup_{n=1}^\infty A_n$. Denote $A_0=\emptyset$ and $B_n=A_n \setminus A_{n-1}$ for $n \in \N$. Fix a sequence $(b_n)_{n=1}^\infty$ in $(0,\infty)$ such that $\sum_{n=1}^\infty b_n\mu(B_n)<\infty$ and let $f=\sum_{n=1}^\infty b_n \one_{B_n}$. Then $f \in L_1(\mu)$ with $f>0$. The measure $\nu:=f\mu$ is finite and the map $U: g \mapsto g/f$ is a linear isometry from $L_1(\mu)$ onto $L_1(\nu)$. Let $Y=U(Z)$. Since $U$ is an isometry, by Theorem \ref{thm:L-embedded}, it is enough to show that $Y$ is closed in $L_1(\nu)$ for the topology of convergence in measure. So let $(h_n)_n$ be a sequence in $Y$ converging in measure to $h\in L_1(\nu)$. Set $g_n=h_nf \in Z$ and $g=hf \in L_1(\mu)$. We need to show that $g\in Z$, so it is enough to show that $(g_n)_n$ converges locally in measure to $g$. We fix $A \in \Sigma$ such that $\mu(A)<\infty$ and $\eps>0$. Then, for any $K \in \N$,
		\[
		\mu(\{|g_n\one_A -g\one_A|>\eps\}) \le \sum_{k=1}^K \mu(\{b_k|h_n-h|>\eps\})+\sum_{k>K}\mu(A \cap B_k).
		\]
		For a given $\eta>0$, there exists $K\in \N$ such that $\sum_{k>K}\mu(A \cap B_k)<\eta$, because $\mu(A)<\infty$. Then, since  $(h_n)_n$ converges to $h$ in measure, there exists $N\in \N$ such that for all $n\ge N$, $\sum_{k=1}^K \mu(\{b_k|h_n-h|>\eps\}) <\eta$. This shows that $(g_n)_n$ converges locally in measure to $g$ and concludes the proof. 
	\end{proof}
	
	\begin{theorem}[{\cite[Lemme~4]{Godefroy}}]
		Let $Y$ be a Banach space. If $Y$ is $L$-embedded in its bidual, then $Y$ is weakly sequentially complete.  
	\end{theorem}
	
	Our main theorem applies Godefroy's results to our constructions of the preceding subsections.
	
	\begin{theorem}[Main Theorem] \label{thm:main}
		Let $X$ be a proper, geodesic metric space with locally finite $\H^1$-measure. Then $\F(X)$ is $L$-embedded in its bidual, and hence weakly sequentially complete.
	\end{theorem}
	
	\begin{proof}
		Let $\gamma:T \to X$ satisfy the conclusion of Theorem~\ref{thm:isometric}. This theorem then gives us that the pullback map   $\widehat\gamma^*:f\mapsto f\circ \gamma$ is an isometry from $\Lip_0(X)$ into $\Lip_0(T)$. By taking the adjoint of this map,  it follows that $\F(X)$ is isometric to $L_1(T,\lambda_T)/\ker(\widehat{\gamma})$. By Theorem~\ref{thm:kernel}, the norm-closed subspace $\ker(\widehat{\gamma})$ is closed with respect to the topology of local convergence in measure. Then by Corollary~\ref{cor:localconvergence}, $\F(X) = L_1(T,\lambda_T)/\ker(\widehat{\gamma})$ is $L$-embedded in its bidual. 
	\end{proof}
	
	\begin{remark} Let us mention that Theorem \ref{thm:L-embedded} is valid for a slightly larger class of subspaces of $L_1(\mu)$, called ``bien disposés'' in \cite{Godefroy}, or ``nicely placed'' in English, which are the subspaces of $L_1(\mu)$ such that their unit balls are closed for the topology of convergence in measure. It is indeed a larger class, as $H_1(\mathbb T)$ is dense in $L_1(\mathbb T)$ for the topology of convergence in measure, while its unit ball is closed for the same topology (see Example IV.3.6. in \cite{HWW}). We also refer to \cite[Section VII.5]{BaireGodefroy} for a presentation of this subject. 
	\end{remark}
	
	\begin{remark}
		In fact, thanks to a result by H. Pfitzner \cite{Pfitzner}, we can say more about $\F(X)$. It is proved in \cite{Pfitzner} that if $Y$ is a separable Banach space which is $L$-embedded in its bidual, then $Y$ has property $(X)$ introduced by Godefroy and Talagrand \cite{GodefroyTalagrand}. Property $(X)$ implies property $(V^*)$ of Pełczyński, which in turm implies weak sequential completeness. We shall not define these properties here. Let us mention that the converses of the above implications are false in general. However it is an open question to know whether they are equivalent for Lipschitz-free spaces. We also recall that Godefroy and Talagrand proved in \cite{GodefroyTalagrand} that a Banach space with property $(X)$ is the unique isometric predual of its dual. Since property $(X)$ is invariant under equivalent renormings, it follows that, under the assumptions of Theorem \ref{thm:main}, $\F(X)$ is the unique isometric predual of its dual, for any equivalent norm on $\F(X)$. 
	\end{remark}

	\section{A remark on a result by C\'uth, Kalenda, and Kaplick\'y}
	
	In \cite {CKK2019} C\'uth, Kalenda, and Kaplick\'y proved the following important result.
	
	\begin{theorem}\emph{(\cite[Theorem 1.1]{CKK2019}).} For any $d\in \N$, the space $\cal F(\ell_2^d)$ is $1$-complemented in $\cal F(\ell_2^d)^{**}$. 
	\end{theorem}
	
	Moreover, as stated in \cite[Theorem 1.1]{CKK2019},  any $d$-dimensional normed space $E$, $\cal F(E)$ is $d_{BM}(E,\ell_d^2)$-complemented into $\cal F(E)^{**}$, where $d_{BM}(E,\ell_d^2)$ denotes the Banach-Mazur distance from $E$ to $\ell_2^d$ and $d_{BM}(E,\ell_d^2)\le \sqrt{d}$. In this section, we present the following improvement.
	
	\begin{theorem}\label{t:C-complementation} Let $E$ be a finite dimensional normed and assume that for some $C\ge 1$, every finite dimensional subspace of $E$ is $C$-complemented in $E$. Then $\cal F(E)$ is $C$-complemented into $\cal F(E)^{**}$.
	\end{theorem}
	
	Theorem~\ref{t:C-complementation} will follow from Theorems~\ref{thm:comp1} and \ref{thm:comp2} below. As the reader will see,  the only truly new ingredient is that under the assumptions of Theorem \ref{t:C-complementation}, there exists a Borel map assigning to each subspace $F$ of $E$ a projection from $E$ onto $F$ with  norm at most $C$ (see Proposition \ref{p:Borel0}). The other parts of the argument will follow the methods of \cite{CKK2019}. 
	
	We start by   stating a straightforward consequence of this Theorem~\ref{t:C-complementation}. Precisely, an easy consequence of the Hahn--Banach Theorem is that if $E$ is a normed space of dimension 2, then every subspace of $E$ is $1$-complemented in $E$. So we have:
	
	\begin{corollary}\label{CorollaryDim2} For any $2$-dimensional normed space $E$, the space $\cal F(E)$ is $1$-complemented in $\cal F(E)^{**}$.
	\end{corollary}
	
	In what follows, given a finite dimensional normed  space $E$,  the set $\cal S(E)$ of subspaces of $E$ will be equipped with its Effros--Borel structure. Let us now state our selection result.
	
	\begin{proposition}\label{p:Borel0} 
		Let $E$ be a finite dimensional normed space and $C\ge 1$ such that any subspace of $E$ is $C$-complemented in $E$. Then, there exists a Borel map ${\tilde{P}}: \cal S(E) \to B(E)$, such that for all $F\in \cal S(E)$, ${\tilde{P}}(F)$ is a projection from $E$ onto $F$ satisfying $\|{\tilde{P}}(F)\|\le C$.
	\end{proposition}
	
	We will need a few preparatory lemmas. Most of them are classical manipulations inspired by \cite{Bossard2002}. We include the details for the reader's convenience and warm up with the following.
	
	\begin{lemma}\label{l:Borel1} Denote 
		$$\cal B:=\left\{(F,e)=(F,(e_n)_{n=1}^\infty) \in {Cl}(E)\times E^\N:  \{e_n,\ n\in \N\}\ \text{is a dense subset of}\ F\right\}.$$
		Then $\cal B$ is a Borel subset of ${Cl}(E)\times E^\N$. 
	\end{lemma}
	
	\begin{proof} Let us fix a countable basis $\cal U$ of open subsets of $E$. Then, for $F\in {Cl}(E)$ and $e=(e_n)_{n=1}^\infty \in E^\N$, we have that $\{e_n: n\in \N\}$ is a dense subset of $F$ if and only if 
		\[(\forall n \in \N,\ e_n\in F)\ \ \text{and}\ (\forall U \in \cal U,\ F \cap U=\emptyset\ \text{or}\ \exists n\in \N\ \text{s.t.} \  e_n\in U).\]
		Since $F$ is closed, we have that 
		\[e_n\in F \Longleftrightarrow \forall U \in \cal U\ (e_n \notin U\ \text{or}\ F \cap U \neq \emptyset).\]
		It follows that $\{(F,e)\in {Cl}(E)\times E^\N : \forall n\in \N\ e_n\in F\}$ is Borel in ${Cl}(E)\times E^\N$. The other condition defining $\cal B$ is clearly Borel in ${Cl}(E)\times E^\N$. This concludes the proof.  
	\end{proof}
	
	The crucial tool for our proof will be the Kuratowski-Ryll-Nardzewski selection theorem. We refer to   \cite[Theorem 12.13]{KechrisBook1995} for its general statement and its proof. We apply it first on ${Cl}(E)$, equipped with its Effros--Borel structure.

	\begin{theorem}[Kuratowski-Ryll-Nardzewski Selection Theorem] 
		There exists a sequence of Borel functions $d_n:{Cl}(E) \to E$ such that for each nonempty $G \in {Cl}(E)$, $\{d_n(G) : n\in \N\}$ is dense in $G$. 
		\label{Selection_theorem}
	\end{theorem}
	
	Let us denote $\cal C=\cal B \cap (\cal S(E) \times E^\N)$. It follows from Lemma \ref{l:Borel1} that $\cal C$ is Borel in $\cal S(E)\times E^\N$. Then, by Theorem \ref{Selection_theorem}, $\delta:\cal S(E) \to \cal C$ defined by $\delta(F)=(F,(d_n(F))_{n=1}^\infty)$, for $F\in \cal S(E)$ is a Borel map.  
	
	\medskip We now denote $K=\{T \in B(E) : T^2=T\ \text{and}\ \|T\|\le C\}$. Since $E$ is finite dimensional, $K$ is a compact subset of $B(E)$. For $F \in \cal S(E)$, we set $K_F=\{T \in K : T(E)=F\}$, which clearly is closed, and therefore compact, in $K$. Recall that, by assumption, $K_F$ is nonempty for all $F\in \cal S(E)$. We also equip ${Cl}(K)$ with its Effros--Borel structure. 
	
	\begin{lemma}\label{l:Borel2}
		The map $\phi: \cal C \to {Cl}(K)$, defined by $\phi(F,e)=K_F$ for $(F,e)\in \cal C$, is Borel. 
	\end{lemma}
	
	\begin{proof} Since  $K$ is compact metrizable, Proposition \ref{p:EffrosBorel_on compact} asserts that the Effros--Borel structure of ${Cl}(K)$ is also generated by the sets $\{F\in {Cl}(K):F\sbs V\}$, where $V$ runs over the open subsets of $K$. Thus it is enough to show that if $A \in {Cl}(K)$, then  $\{(F,e): K_F \cap A \neq \emptyset\}$ is Borel in $\cal C$.
		So let us fix a closed nonempty subset $A$ of $K$, $(T_r)_{r=1}^\infty$ a dense sequence in $A$ and $(x_k)_{k=1}^J$ a basis of $E$. Note that, for $(F,e)\in \cal C$:
		\[K_F=\{T \in K: \forall n\in \N,\ Te_n=e_n\}\cap \{T\in K: \forall k\le J,\ Tx_k \in F\}. \]
		Assume that $K_F \cap A \neq \emptyset$.  Using the density of $(T_r)_{r=1}^\infty$ in $A$, we deduce that 
		\begin{equation}\label{e:Borel3}
			\forall i \in \N\ \exists r=r_i\in \N:
		\end{equation}
		\[\forall n\in \N\ \|T_re_n-e_n\|\le 2^{-i}\|e_n\|\ \text{and}\  \forall k\le J,\ d(T_r(x_k),F)\le 2^{-i}\|x_k\|.\]
		This is actually an equivalence. Indeed, assume (\ref{e:Borel3}) is satisfied. Using the compactness of $A$, we can find a subsequence of $(T_{r_i})_i$ converging in operator norm to some $T \in A$. Then $Te_n=e_n$ for all $n\in \N$ and $Tx_k \in F$ for all $ k\le J$. Due to the local compactness of $E$ and $F$, the condition $d(T_r(x_k),F)\le 2^{-i}\|x_k\|$ is equivalent to $F\cap \overline{B}(T_r(x_k),2^{-i}\|x_k\|)\neq \emptyset$, which, by the remark at the beginning of the proof, is a Borel condition in ${Cl}(K)$. It should now be clear that Equation (\ref{e:Borel3}) defines a Borel subset of $\cal C$. 
	\end{proof}

	\begin{proof}[Proof of Proposition \ref{p:Borel0}]
		It also follows from the Kuratowski-Ryll-Nardzewski Selection Theorem that there exists a Borel map $\rho:{Cl}(K) \to K$ such that for all nonempty $G$ in ${Cl}(K)$, $\rho(G) \in G$. Define now ${\tilde{P}}=\rho \circ \phi \circ \delta: \cal S (E) \to B(E)$. It follows from our discussion that ${\tilde{P}}$ is Borel and for all $F\in \cal S(E)$, ${\tilde{P}}(F) \in K_F$. This concludes our proof. 
	\end{proof}

	In order to explain how Proposition \ref{p:Borel0} allows us to deduce Theorem \ref{t:C-complementation} we need to recall the proof technique of \cite{CKK2019}. Let $d\in\N$, $\lambda$ be the Lebesgue measure on $\R^d$, and let $E$ denote $\R^d$ equipped with a norm. Let $\M(\R^d,E)$ denote the vector space of $E$-valued finitely additive measures on $\R^d$, equipped with the total variation norm. Following \cite[pages~414, 420, 421, 427]{CKK2019}, we define the following Banach spaces, which, apart from $X(E)$, are closed subspaces of $\M(\R^d,E)$:
	\begin{align*}
		\M_\sigma(\R^d,E) &= \{\mu\in\M(\R^d,E): \mu \text{ is } \sigma\text{-additive}\} \\
		\M_{ac}(\R^d,E) &= \{\mu\in\M(\R^d,E): \forall B\in\B(\R^d), \lambda(B)=0 \Rightarrow \mu(B)=0\} \\
		Y(E) &= \{g\in L_1(\R^d,E): \mathrm{div}\,g = 0\} \\
		X(E)&=Y(E)^\perp \sbs L^\infty(\R^d,E^*),\ \text{where}\ L_1(\R^n,E)^*\ \text{is identified with } L^\infty(\R^d,E^*)\\
		Y_\sigma(E) &= \{\mu\in\M_\sigma(\R^d,E): \mathrm{div}\,\mu = 0\} \\
		\M_\sigma^{\mathrm{div}}(E) &= (\M_\sigma(\R^d,E) \cap \M_{ac}(\R^d,E)) + Y_\sigma(E)
	\end{align*}
	The space $L_1(\R^d,E)$ can be isometrically identified with $\M_{ac}(\R^d,E)\cap\M_\sigma(\R^d,E)$ via $f \mapsto f\lambda$ (see \cite[page~422]{CKK2019}). We also recall the restriction map $R\colon \M(\R^d,E) \to \M_\sigma(\R^d,E)$ defined to be the adjoint of the inclusion $C_0(\R^d,E^*) \to \B_b(\R^d,E^*)$ (\cite[pages~416, 420]{CKK2019}).
	
	We recall results from \cite{CKK2019}:
	
	\begin{proposition}\emph{(\cite[ Proposition~6.2 and Lemma~6.5]{CKK2019}).}
		\begin{itemize}
			\item The canonical identity map $J\colon L_1(\R^d,E)/Y(E) \to \M_\sigma^{\mathrm{div}}(E)/Y_\sigma(E)$ given by \\ 
			$J(f+Y(E)) = f\lambda + Y_\sigma(E)$ is a surjective isometry.
			\item $R(X(E)^\perp) = Y_\sigma(E)$  
		\end{itemize}\label{prop:CKKresults}
	\end{proposition}
	
	The following is a general statement that can be extracted from the proof of \cite[Proposition~8.1]{CKK2019}.
	
	\begin{theorem} \label{thm:comp1}
		Suppose that there exists a constant $C<\infty$ and a projection $P\colon  \M_\sigma(\R^d,E) \to \M_\sigma^{\mathrm{div}}(E)$ with $\|P\|\leq C$. Then $L_1(\R^d,E)/Y(E)$ is $C$-complemented in $\M_{ac}(\R^d,E)/X(E)^\perp$.
	\end{theorem}
	
	\begin{proof}
		The proof is the same as that of \cite[Proposition~8.1]{CKK2019}. Here are the details:
		
		By Proposition~\ref{prop:CKKresults} and the fact that $P(Y_\sigma(E)) = Y_\sigma(E)$ since $P$ is a projection onto $M_\sigma^{\mathrm{div}}(E) \sps Y_\sigma(E)$, we have $P(R(X(E)^\perp)) =Y_\sigma(E)$, which implies that $PR$ descends to a well-defined linear map $V\colon \M_{ac}(\R^d,E)/X(E)^\perp \to \M_\sigma^{\mathrm{div}}(E)/Y_\sigma(E)$ of norm at most $  C$. Then by Proposition~\ref{prop:CKKresults}, $J^{-1}V: \M_{ac}(\R^d,E)/X(E)^\perp \to L_1(\R^d,E)/Y(E)$ has norm at most $C$ and it is a projection onto $L_1(\R^d,E)/Y(E)$ because for any $f \in L_1(\R^d,E)$,
		\[J^{-1}(V(f + Y(E)) = J^{-1}(f\lambda + Y_\sigma(E)) = f+Y(E).\qedhere\]
	\end{proof}
	
	Therefore, in order to prove Theorem~\ref{t:C-complementation}, it remains to prove that there exists a projection $P\colon \M_\sigma(\R^d,E) \to \M_\sigma^{\mathrm{div}}(E)$ with $\|P\|\leq C$. Towards this end, we need to recall another result from \cite{CKK2019}.
	
	\begin{proposition}\emph{(\cite[Proposition~7.1]{CKK2019}).} 
		Let $\nu \in \M_\sigma(E)$ be a nonnegative measure. Then there exists a mapping $T_\nu$ assigning to each $x\in\R^d$ a vector subspace $T_\nu(x) \sbs E$ with the following properties:
		\begin{itemize}
			\item $T_\nu$ is lower $\nu$-measurable, i.e., $\{x\in\R^d: T_\nu(x) \cap G \neq \emptyset\}$ is $\nu$-measurable for any $G \sbs E$ open.
			\item For any $f \in L_1(\nu,E)$, we have
			\begin{equation*}
				f\nu \in \M_\sigma^{\mathrm{div}}(E) \Leftrightarrow f(x) \in T_\nu(x) \hspace{.15in} \text{ for } \nu\text{-almost all } x\in\R^d.
			\end{equation*}
		\end{itemize}\label{prop:CKK7.1}
	\end{proposition}
	
	The following theorem is the last ingredient needed to prove Theorem~\ref{t:C-complementation}.
	
	\begin{theorem} \label{thm:comp2}
		Suppose that there exists a constant $C<\infty$ such that every subspace of $E$ is $C$-complemented. Then there exists a projection $P\colon \M_\sigma(\R^d,E) \to \M_\sigma^{\mathrm{div}}(E)$ with $\|P\|\leq C$.
	\end{theorem}
	
	\begin{proof}
		The proof is just the same as in \cite[Propositions~7.2 and 7.4]{CKK2019}. The only difference is that we use the measurably selected projections $\tilde{P}(F)$ of Proposition~\ref{p:Borel0} instead of orthogonal projections arising from a $C$-equivalent Euclidean norm. Here are the details:
		
		For each $\nu\in\M_\sigma(E)$, let $T_\nu$ be the mapping provided by Proposition~\ref{prop:CKK7.1}, and define $\tilde{P}_\nu\colon L_1(\nu,E) \to L_1(\nu,E)$ by
		\begin{equation*}
			\tilde{P}_\nu(f)(x) = \tilde{P}(T_\nu(x))(f(x)),
		\end{equation*}
		where $\tilde{P}\colon \mathcal{S}(E) \to B(E)$ is a map satisfying the conclusion of Proposition~\ref{p:Borel0}. For $\tilde{P}_\nu$ to be well-defined, we need to verify $\nu$-measurability of $\tilde{P}_\nu(f)$. Let $\nu\in\M_\sigma(E)$ and $f\in L_1(\nu,E)$. Then notice that $\tilde{P}_\nu(f)$ is the composition of the mappings
		\begin{equation*}
			\R^d \overset{T_\nu \times f}\to \cal S (E) \times E \overset{\tilde{P} \times id}\to B(E)\times E \to E,
		\end{equation*}
		where the last map is the evaluation $(L,x) \mapsto L(x)$. When $\R^d$ is equipped with the $\sigma$-algebra of $\nu$-measurable sets and the other spaces are equipped with their Borel $\sigma$-algebras, each of the maps in this composition are obviously measurable, and hence so is $\tilde{P}_\nu(f)$. With this established, we clearly have that $\tilde{P}_\nu$ is a linear operator on $L_1(\nu,E)$ with norm at most $C$, and also that $\tilde{P}_\nu(hf) = h\tilde{P}_\nu(f)$ whenever $h \in L^\infty(\nu)$.
		
		We also clearly have that $\tilde{P}_\nu(f)\nu \in \M_\sigma^{\mathrm{div}}(E)$ for every $\nu\in\M_\sigma(E)$ and every $f \in L_1(\nu,E)$ by definition of $\tilde{P}$ and $T_\nu$. Then we define $P: M_\sigma(\R^d,E) \to \M_\sigma^{\mathrm{div}}(E)$ by
		\begin{equation*}
			P(\mu) = \tilde{P}_\nu(f)\nu \hspace{.15in} \text{if } \mu = f\nu \, \text{ where } \nu\in\M_\sigma(\R^d),\, \nu\geq 0,\, f\in L_1(\nu,E).
		\end{equation*}
		We first check that $P$ is well-defined using \cite[Lemma~7.3]{CKK2019}, which states that
		\begin{equation} \label{eq:CKK7.3}
			\nu_1 \ll \nu_2 \Rightarrow T_{\nu_1}(x) = T_{\nu_2}(x) \text{ for } \nu_1\text{-a.e. } x\in\R^d.
		\end{equation}
		The argument is just a repetition of that on \cite[page~437]{CKK2019}. Let $\nu_1,\nu_2 \in \M_\sigma(\R^d)$ with $\nu_1,\nu_2 \geq 0$, and let $f_1\in L_1(\nu_1,E)$ and $f_2\in L_1(\nu_2,E)$. Let $h_1,h_2 \in L^\infty(\nu_1+\nu_2)$ be the Radon-Nikod\'ym derivatives of $\nu_1,\nu_2$, respectively, with respect to $\nu_1+\nu_2$. Assume that $f_1\nu_1 = f_2\nu_2$, which implies
		\begin{equation} \label{eq:h1f1=h2f2}
			(h_1f_1)(x) = (h_2f_2)(x) \text{ for } (\nu_1+\nu_2)\text{-a.e. } x\in\R^d.
		\end{equation}
		Then we have
		\begin{align*}
			\tilde{P}_{\nu_1}(f_1)\nu_1 &\overset{\eqref{eq:CKK7.3}}{=} \tilde{P}_{\nu_1+\nu_2}(f_1)\nu_1 = \tilde{P}_{\nu_1+\nu_2}(f_1)h_1(\nu_1+\nu_2) = \tilde{P}_{\nu_1+\nu_2}(h_1f_1)(\nu_1+\nu_2) \\
			&\overset{\eqref{eq:h1f1=h2f2}}{=} \tilde{P}_{\nu_1+\nu_2}(h_2f_2)(\nu_1+\nu_2) = \tilde{P}_{\nu_1+\nu_2}(f_2)h_2(\nu_1+\nu_2) = \tilde{P}_{\nu_1+\nu_2}(f_2)\nu_2 \\
			&\overset{\eqref{eq:CKK7.3}}{=} \tilde{P}_{\nu_2}(f_2)\nu_2.
		\end{align*}
		This proves that $P$ is well-defined.
		
		The fact that $P$ is linear follows from a similar argument to the one above and indeed is just a repetition of the argument on \cite[page~438]{CKK2019}. We omit the details. That $\|P\|\leq C$ and that $P$ is a projection onto $\M_\sigma^{\mathrm{div}}(E)$ are immediately deducible from the definitions.
	\end{proof}
	\begin{proof}[Proof of Theorem~\ref{t:C-complementation}]
		By Theorems~\ref{thm:comp1} and \ref{thm:comp2}, $L_1(\R^d,E)/Y(E)$ is $C$-complemented in $\M_{ac}(\R^d,E)/X(E)^\perp$. By \cite[page~421]{CKK2019}, $\F(E)^{**}$ is linearly isometrically identified with $\M_{ac}(\R^d,E)/X(E)^\perp$ in such a way that $\F(E)$ becomes identified with $L_1(\R^d,E)/Y(E)$. Hence, $\F(E)$ is $C$-complemented in $\F(E)^{**}$.
	\end{proof}

	\noindent
	\textbf{Acknowledgements.} We wish to thank Gilles Godefroy for insightful discussions on nicely placed subspaces of $L_1$.
	
	This paper was written under the auspices of the American Institute of Mathematics (AIM) SQuaREs program as part of the “Nonlinear geometry of Banach spaces” SQuaRE project.


\begin{thebibliography}{99}
		
		\bibitem{BaudierLancienBook} F. Baudier and G. Lancien, \textit{Asymptotic and nonlinear geometries of Banach spaces and their interactions}, Cours Spécialisés, \textbf{33}, Soc. Math. France, Paris, (2026).
		
		\bibitem{Bossard2002}
		B. Bossard,
		\textit{A coding of separable Banach space. Analytic and coanalytic families of Banach spaces},
		Fundam. Math. \textbf{172} (2002), 117--151.
		
		\bibitem{BGLPPS2026} B. Braga, C. Gartland, G. Lancien, P. Motakis, E. Perneck\'{a} and Th. Schlumprecht, \textit{On the $L$-embeddability of Lipschitz-free spaces}, preprint arXiv: https://arxiv.org/pdf/2610.08616
		
		\bibitem{CKK2019}
		M. C{\'u}th, O. Kalenda and P. and Kaplick{\'y}, \textit{Finitely additive measures and complementability of {Lipschitz}-free spaces}, Isr. J. Math., \textbf{230} (2019), no. 1, 409--442.
		
		\bibitem{Golab} A. Giovanni and M. Ottolini,
		\textit{On the structure of continua with finite length and Golab's semicontinuity theorem}, Nonlinear Anal., \textbf{153} (2017), 35--55. 
		
		\bibitem{Godard2010PAMS} A. Godard, \textit{Tree metrics and their Lipschitz-free spaces}, Proc. Am. Math. Soc., \textbf{138} (2010), no. 12, 4311--4320. 
		
		\bibitem{Godefroy} G. Godefroy, \textit{Sous-espaces bien dispos\'es de $L^1$ - Applications}, Trans.  Amer. Math. Soc., \textbf{286} (1984), no. 1, 227--249. 
		
		\bibitem{BaireGodefroy} G. Godefroy, \textit{Applications of Baire methods}, Cours Spécialisés, Soc. Math. France, Paris, to appear.
		
		\bibitem{GodefroyKalton2003} G. Godefroy and N.\,J. Kalton, \textit{Lipschitz-free Banach spaces}, Studia Math., {\bf 59} (2003), 121--141.
		
		\bibitem{GodefroyTalagrand} G. Godefroy and M. Talagrand,
		\textit{Nouvelles classes d'espaces de {Banach} {\`a} pr{\'e}dual unique}, S\'eminaire
		d’Analyse Fonctionnelle de l’\'Ecole Polytechnique, (1980-81).
		
		\bibitem{HWW} P. Harmand, D. Werner and W. Werner, \textit{$M$-ideals in Banach spaces and Banach algebras}, Lecture Notes in Mathematics, vol. 1547, Springer-Verlag, Berlin, 1993. 
		
		\bibitem{KechrisBook1995}
		A. Kechris, \textit{Classical descriptive set theory}, Grad. Texts Math., Springer-Verlag, Berlin, 1995. 
		
		\bibitem{Kirchheim} B. Kirchheim, \textit{Rectifiable metric spaces: local structure and regularity of the Hausdorff measure}, Proc. Amer. Math. Soc., \textbf{121} (1994), no. 1, 113--123.
		
		
		\bibitem{Mattila_Book} P. Mattila, \textit{Geometry of sets and measures in {E}uclidean spaces}, Cambridge Studies in Advanced Mathematics, vol. {44}, Cambridge University Press, Cambridge, 1995.
		
		\bibitem{Pfitzner} H. Pfitzner, \textit{Separable {L}-embedded {Banach} spaces are unique preduals}, {Bull. Lond. Math. Soc.}, \textbf{39} (2007), no. 6, 1039--1044.
		
		
		\bibitem{Waz} T. Wazewski, \textit{Kontinua prostowalne w zwiazku z funkcjami i odwzorowaniami absolutnie ciaglemi}, Ann. Soc. Polon. Math., \textbf{3} (1927), 9--49. 
		
		
		\bibitem{Weaver2} N. Weaver, \textit{Lipschitz algebras 2nd ed.}, World Scientific Publishing Co., River Edge, NJ, 2018.
		
		
	\end{thebibliography}
\end{document}